\documentclass[review]{elsarticle}
\usepackage{lineno,hyperref}
\usepackage{mathrsfs}
\usepackage{amsfonts}
\usepackage{amssymb}
\usepackage{amsmath}
\usepackage{hyperref}
\usepackage{amsthm}
\usepackage{cases}
\usepackage{arydshln}
\usepackage[usenames,dvipsnames]{color}
\usepackage{graphicx}%
\newtheorem{thm}{Theorem}[section]
\newtheorem{cor}[thm]{Corollary}
\newtheorem{lemma}[thm]{Lemma}

\newtheorem{rem}[thm]{Remark}
\newtheorem{eg}[thm]{Example}

\modulolinenumbers[5]

\journal{LAA}

\begin{document}

\begin{frontmatter}

\title{Revisiting a sharpened version of Hadamard's determinant inequality}
 
\author{Minghua Lin\fnref{myfootnote}}
\address{School of Mathematics and Statistics, Xi'an Jiaotong University, Xi'an 710049, China}
\fntext[myfootnote]{Email: mh.lin@xjtu.edu.cn}

 \author{Gord Sinnamon \fnref{gs}}
 \address{Department of Mathematics, University of Western Ontario, London, N6A 5B7, ON, Canada}
\fntext[gs]{Email: sinnamon@uwo.ca}

\begin{abstract}  Hadamard's determinant inequality  was refined and generalized by Zhang and Yang in [Acta Math. Appl. Sinica 20 (1997) 269-274]. Some special cases of the result were rediscovered recently by Rozanski,   Witula and Hetmaniok in [Linear Algebra Appl. 532 (2017) 500-511].  We revisit the result in the case of positive semidefinite matrices, giving a new proof in terms of majorization and a complete description of the conditions for equality in the positive definite case.  We also mention a block extension, which makes use of a result of Thompson in the 1960s. 
\end{abstract}


\end{frontmatter}


\section{Introduction} 
 
 Perhaps the best known determinantal inequality in mathematical sciences is the Hadamard inequality  (e.g., \cite[p. 505]{HJ13}) which says that 
 if $A=(a_{ij})$ is an $n\times n$ (Hermitian) positive definite matrix, then 
 \begin{eqnarray}\label{h1} \det A\le a_{11}\cdots a_{nn},\end{eqnarray} 
 and equality holds if and only if $A$ is diagonal.  
  

Two decades ago, Zhang and Yang obtained an elegant sharpening of the Hadamard inequality.  They proved it for a more general class of matrices and included a term involving off-diagonal entries of the matrix.

Let $A$ be an $n\times n$ complex matrix. For   a non-empty proper subset $G$ of $\{1,\dots,n\}$, let $A[G]$ denote the principal submatrix of $A$ formed by discarding the $i$th row and column of $A$, for each $i\notin G$. We say $A$ is an $F$-matrix if for all such $G$, $\det(A[G])\ge0$ and $\det(A)\le\det(A[G])\det(A[G^C])$ That is, all principal minors of $A$ are non-negative and they satisfy a Fischer-type inequality. Standard results show that the Hermitian $F$-matrices are exactly the positive semi-definite matrices.

 \begin{thm}\label{ZYmain} (Zhang-Yang \cite{ZY97}) If $A=(a_{ij})$ is an $F$-matrix, then $a_{ij}a_{ji}\ge0$ for all $i,j$ and, if $\sigma$ is a non-trivial permutation of $\{1,\dots,n\}$, then
\begin{equation}\label{maxZY}
\det(A)+\Big(\prod_{i=1}^na_{i,\sigma(i)}a_{\sigma(i),i}\Big)^{1/2}\le \prod_{i=1}^na_{ii}.
\end{equation}
\end{thm}
 
Without noticing the work of  Zhang and Yang  in \cite{ZY97}, which was written in Chinese, Rozanski, Witula and Hetmaniok recently rediscovered some special cases of (\ref{maxZY}) in \cite{RWH2017}. For this reason, we think it worthwhile to bring the nice result of Theorem \ref{ZYmain} to the attention of the linear algebra community.

Besides Theorem \ref{ZYmain}, \cite{ZY97} also contains necessary and sufficient conditions for an $F$-matrix $A$ to satisfy the equation $\det(A)=a_{11}\dots a_{nn}$. The same was done for the equation $\operatorname{permanent}(A)=a_{11}\dots a_{nn}$. However, Zhang and Yang did not give conditions for equality to hold in (\ref{maxZY}). Conditions for equality were included by Rozanski, et al. for the special cases considered in \cite{RWH2017}.

There are  multiple known ways to prove  the Hadamard inequality (\ref{h1}) in the literature (see, e.g.,  \cite{Hol07, HJ13, MO82}). We recall that one insightful way of seeing the Hadamard  inequality is via Schur's majorization inequality (see Lemma \ref{lem1}). For a quick summary of the intimate connection between majorization and determinant inequalities, we refer to \cite{Lin17}. 

Our initial motivation was that since the original Hadamard inequality (\ref{h1}) is immediate from   majorization (see, e.g.  \cite[p. 44]{And94}, \cite[p. 67]{Zha13}), it would be nice if Theorem \ref{ZYmain}  could also be seen from that perspective.  In this note, we give a new proof of Theorem \ref{ZYmain} for positive semi-definite matrices using majorization techniques. Our results include necessary and sufficient conditions for equality to hold when the matrix $A$ is positive definite.
 
Before proceeding, let us fix some notation. For a vector $x\in\mathbb{R}^n$, we denote by $x^{\downarrow}=(x_1^{\downarrow}, \ldots, x_n^{\downarrow})\in\mathbb{R}^n$ the vector with the same components as $x$, but sorted in nonincreasing order. Given $x, y\in \mathbb{R}^n$, we say that $x$  majorizes $y$ (or $y$ is majorized by $x$),  written as $x\succ y$, if
 $$
\sum_{i=1}^k x_i^{\downarrow} \geq \sum_{i=1}^k y_i^{\downarrow} \quad \text{for } k=1,\dots, n-1
$$ 
and equality holds at $k=n$.   
   
Three basic facts about majorization are given below. The first is a matrix characterization, the next is Schur's majorization inequality, and the last is a consequence for the elementary symmetric functions. A {\it doubly stochastic matrix} is square matrix with non-negative entries and all row and column  sums equal to 1.

\begin{lemma}\label{lem0} \cite[p. 253]{HJ13} If $x$ and $y$ are real row vectors then $x$ majorizes $y$ if and only if there exists a doubly stochastic matrix $S$ such that $y=xS$. 
\end{lemma}

\begin{lemma}\label{lem1} \cite[p. 249]{HJ13}  The eigenvalues  of a Hermitian matrix majorize its diagonal entries. \end{lemma}
  
Fix a positive integer $n$ and let $e_k(x)$, $k=1, 2, \ldots, n$, denote the $k$th elementary symmetric function in the $n$ variables $x_1, \ldots, x_n$. See \cite[p.114]{MOA11}. By convention, $e_0(x)=1$. 

\begin{lemma}\label{lem2} \cite[p.115]{MOA11}  Let $x, y\in [0,\infty)^n$. If $n\ge2$ and $x\succ y$, then $e_k(x)\le e_k(y)$ for $k=0, 1, \ldots, n$. If $k>1$ and $x, y\in (0,\infty)^n$ then equality holds if and only if $x^\downarrow=y^\downarrow$.  \end{lemma}
  
 \section{A new proof of  Theorem \ref{ZYmain} and more} 

In this section, $\Lambda$, $V$ and $B$ will be as follows: Let $\Lambda$ be a diagonal matrix with non-negative diagonal entries $\lambda_1,\dots,\lambda _n$. Let $V=(v_{ij})$ be an $n\times n$ matrix whose rows and columns are all unit vectors, that is, all diagonal entries of $V^*V$ and $VV^*$ are equal to $1$. Set $B=(b_{ij})=V^*\Lambda V$. 
It is important to point out that, in general, $\prod_{i=1}^n\lambda_i\ne \det(B)$, although they are equal when $V$ is unitary.

\begin{lemma}\label{2.L} Let $P(t)=\prod_{i=1}^n(\lambda_{i}-t)$ and $Q(t)=\prod_{i=1}^n(b_{ii}-t)$. Fix $s<t\le\min(\lambda_1,\dots,\lambda_n)$. Then $P(t)\le Q(t)$ and $P(s)-P(t)\le Q(s)-Q(t)$. If $n\ge3$ and $P(s)-P(t)= Q(s)-Q(t)$ then $(b_{11},\dots,b_{nn})$ is a permutation of $(\lambda_1,\dots,\lambda_n)$.
\end{lemma}
\begin{proof} 
The conditions on $V$ show that the matrix $S=(|v_{ij}|^2)$ is doubly stochastic and a calculation shows that $(b_{11}-t,\dots,b_{nn}-t)=(\lambda_1-t,\dots,\lambda _n-t)S$.
By Lemma \ref{lem0} and Lemma \ref{lem2}, 
$$
e_k(\lambda_1-t,\dots,\lambda _n-t)\le e_k(b_{11}-t,\dots,b_{nn}-t)
$$
for $k=0,\dots,n$. In particular,
$$
P(t)=e_n(\lambda_1-t,\dots,\lambda _n-t)\le e_n(b_{11}-t,\dots,b_{nn}-t)=Q(t).
$$
Also,
\begin{align*}
P(s)-P(t)&=\prod_{i=1}^n(\lambda_i-t+t-s)-\prod_{i=1}^n(\lambda_i-t)\\
&=\sum_{k=0}^{n-1}e_k(\lambda_1-t,\dots,\lambda _n-t)(t-s)^{n-k}\\
&\le\sum_{k=0}^{n-1}e_k(b_{11}-t,\dots,b_{nn}-t)(t-s)^{n-k} \\
&=\prod_{i=1}^n(b_{ii}-t+t-s)- \prod_{i=1}^n(b_{ii}-t)
=Q(s)-Q(t).
\end{align*}
If $n\ge3$ and $P(s)-P(t)= Q(s)-Q(t)$, then the above estimate reduces to equality throughout, which implies $e_2(\lambda_1-t,\dots,\lambda _n-t)=e_2(b_{11}-t,\dots,b_{nn}-t)$. But $e_2$ is strictly Schur concave on all of $\Bbb R^n$. (See \cite[A.4]{MOA11} to prove concavity and then \cite[A.3.a]{MOA11} to prove strict concavity.) Thus $(b_{11}-t,\dots,b_{nn}-t)$ is a permutation of $(\lambda_1-t,\dots,\lambda_n-t)$ and therefore $(b_{11},\dots,b_{nn})$ is a permutation of $(\lambda_1,\dots,\lambda_n)$.
\end{proof}

\begin{thm} If $\lambda_1,\dots,\lambda_n$ are non-negative, then
\begin{equation}\label{H+}
\prod_{i=1}^n\lambda_i\le\prod_{i=1}^n b_{ii}.
\end{equation}
If $n\ge3$ and $\lambda_1,\dots,\lambda_n$ are strictly positive and distinct then equality holds if and only if $V$ has exactly one non-zero entry in each row and column.
\end{thm}
\begin{proof} Lemma \ref{2.L} shows that $P(0)\le Q(0)$, which is (\ref{H+}).  Note that $V$ has exactly one non-zero entry in each row and column if and only if $S=(|v_{ij}|^2)$ is a permutation matrix. In that case, since $(b_{11},\dots,b_{nn})=(\lambda_1,\dots,\lambda _n)S$, (\ref{H+}) holds with equality.

Now suppose $\lambda_1,\dots,\lambda_n$ are strictly positive and distinct. If equality holds in (\ref{H+}), that is, if $e_n(\lambda_1,\dots,\lambda _n)=e_n(b_{11},\dots,b_{nn})$, then $b_{11},\dots,b_{nn}$ are also strictly positive. Therefore Lemma \ref{lem2} shows $(b_{11},\dots,b_{nn})$ is a permutation of $(\lambda_1,\dots,\lambda _n)$. It follows that there are permutation matrices $R$ and $R'$ such that $(\lambda_1^\downarrow,\dots,\lambda _n^\downarrow)=(\lambda_1^\downarrow,\dots,\lambda _n^\downarrow)R'SR$. Clearly, $R'SR=(t_{ij})$ is also a doubly stochastic matrix. 

Assume that for some $i$, there exists an $m<i$ such that $t_{im}\ne0$. For this $i$ choose the largest such $m$. Then 
$t_{ij}(\lambda_j^\downarrow-\lambda_i^\downarrow)=0$ for $j>m$, $t_{ij}(\lambda_j^\downarrow-\lambda_i^\downarrow)\ge0$ for $j<m$, and $t_{im}(\lambda_m^\downarrow-\lambda_i^\downarrow)>0$. This shows that
$\sum_{j=1}^nt_{ij}(\lambda_j^\downarrow-\lambda_i^\downarrow)>0$, which contradicts $(\lambda_1^\downarrow,\dots,\lambda _n^\downarrow)=(\lambda_1^\downarrow,\dots,\lambda _n^\downarrow)R'SR$. We conclude that $R'PR$ is upper triangular. It is easy to see that the only upper triangular, doubly stochastic matrix is $I$. Thus, $S$ is a permutation matrix. 
\end{proof}

The following examples show that the conditions for equality may change in the non-generic cases where the $\lambda_1,\dots,\lambda_n$ are not distinct or include one or more zeros.

\begin{eg} Let $\lambda_1=\lambda_2=1$ and $V=\left(\begin{smallmatrix}i\cos\theta&i\sin\theta&0\\\sin\theta&\cos\theta&0\\0&0&-1\end{smallmatrix}\right)$ for some $\theta$. Then $B=V^*V=\left(\begin{smallmatrix}1&\sin2\theta&0\\\sin2\theta&1&0\\0&0&1\end{smallmatrix}\right)$, but $VV^*=\left(\begin{smallmatrix}1&i\sin2\theta&0\\-i\sin2\theta&1&0\\0&0&1\end{smallmatrix}\right)$. So we have equality in (\ref{H+}) but $V$ does not have only one non-zero entry in each row and column. Significantly, the doubly stochastic matrix $S=\left(\begin{smallmatrix}\cos^2\theta&\sin^2\theta&0\\\sin^2\theta&\cos^2\theta&0\\0&0&1\end{smallmatrix}\right)$ is not uniquely determined; it varies with $\theta$.
\end{eg}

\begin{eg} With $\Lambda$ and $V$ as defined above, let $\Lambda'=\left(\begin{smallmatrix}\Lambda&0\\0&0\end{smallmatrix}\right)$ and $V'=\left(\begin{smallmatrix}V&0\\0&1\end{smallmatrix}\right)$. Then we have equality in (\ref{H+}) because both sides are zero, but $V'$ does not have one non-zero entry in each row and column unless $V$ does. 
\end{eg}

The conditions imposed on $V$, above, are satisfied by any unitary matrix but the unitaries are only a small subclass of the possible matrices $V$. The next example gives large class of matrices $V$ that are not unitary.

\begin{eg}
Suppose $T=(t_{ij})$  is an $n\times n$ Hermitian matrix with spectral norm $\|T\|\le 1$ satisfying $t_{ii}=0$ for all $i$. Since $\|T\|\le1$, $I+T$ is positive semi-definite and therefore has a positive semi-definite square root $V=(I+T)^{1/2}$. Note that $V^*V=VV^*=I+T$, a matrix with ones on the diagonal. If $T$ is not zero the matrix $V$ is not unitary.
\end{eg}

On the other hand, if $V$ is unitary, the conditions on $V$ are satisfied automatically, the product $\lambda_1\cdots\lambda_n$ is the determinant of $B$, and $B$ could be any positive semi-definite matrix. 

Let $S_n$ be the group of permutations of $\{1,\dots,n\}$ and let $D_n$ denote the
collection of permutations that have no fixed point. These are the so-called
derangements of $\{1,\dots,n\}$. We also let $e_i$ be the $i$th column of the $n\times n$ identity matrix so that $Ae_i$ extracts the $i$th column of $A$. We remind the reader not to confuse $e_i$ with $e_i(x)$.

 \begin{thm}\label{refined} Let $n\ge 3$. If $A=(a_{ij})$ is positive semi-definite and $\tau\in D_n$, then
 	\begin{equation}\label{maxtau}
 	\det(A)+\prod_{i=1}^n|a_{i,\tau(i)}|\le \prod_{i=1}^na_{ii}.
 	\end{equation}
 	Equality holds if and only if $A$ is diagonal or the two vectors $Ae_i$ and $Ae_{\tau(i)}$ are collinear for each $i$.
 \end{thm}
\begin{proof} Choose a unitary $V$ such that $A=V^*\Lambda V$,  where $\lambda_1,\dots,\lambda_n$ are the (necessarily non-negative) eigenvalues of $A$. This makes $B=A$. Then set $t=\min(\lambda_1,\dots,\lambda_n)$. Lemma \ref{2.L} shows that $P(0)-P(t)\le Q(0)-Q(t)$. The choice of $t$ ensures that $P(t)=0$. Also $P(0)=\det(A)$ and $Q(0)=\prod_{i=1}^na_{ii}$. To prove (\ref{maxtau}) we use the Cauchy-Schwarz inequality to show that $Q(t)\ge\prod_{i=1}^n|a_{i,\tau(i)}|$ for each $\tau\in D_n$. Fix $\tau\in D_n$ and observe that
\begin{equation}\label{CS}\begin{aligned}
|a_{i,\tau(i)}|^2&=\bigg|\sum_{j=1}^n\bar v_{ji}\lambda_jv_{j,\tau(i)}\bigg|^2\\
&=\bigg|\sum_{j=1}^n\bar v_{ji}(\lambda_j-t)v_{j,\tau(i)}\bigg|^2\\
&\le\sum_{j=1}^n|v_{ji}|^2(\lambda_j-t)\sum_{j=1}^n|v_{j,\tau(i)}|^2(\lambda_j-t)\\
&=(a_{ii}-t)(a_{\tau(i),\tau(i)}-t).
\end{aligned}\end{equation}
Thus,
\begin{equation}\label{aCS}
\prod_{i=1}^n|a_{i,\tau(i)}|
\le\bigg(\prod_{i=1}^n(a_{ii}-t)\prod_{i=1}^n(a_{\tau(i),\tau(i)}-t)\bigg)^{1/2}=\prod_{i=1}^n(a_{ii}-t)=Q(t).
\end{equation}
Next we consider conditions for equality. If $A$ is diagonal it is clear that (\ref{maxtau}) holds with equality. Suppose   the two vectors $Ae_i$ and $Ae_{\tau(i)}$ are collinear for each $i$. Then $A$ is singular, $\det(A)=0$, and $t=0$. Fix $i$ and choose $a$ and $b$, not both zero, such that $aAe_i=bAe_{\tau(i)}$. Then for each $j$, $e_j^*VA(ae_i-be_{\tau(i)})=0$. But $AV^*e_j=V^*\Lambda e_j=\lambda_jV^*e_j$ so $\lambda_je_j^*V(ae_i-be_{\tau(i)})=0$. Therefore $av_{ji}=bv_{j,\tau(i)}$ for all $j$ such that $\lambda_j\ne0$. This gives equality in (\ref{CS}). Since this holds for all $i$, we have equality in (\ref{aCS}) as well. Since $\det(A)=0$ we also have equality in (\ref{maxtau}).

Conversely,  suppose equality holds in (\ref{maxtau}). The above proof shows that we must have $P(0)-P(t)= Q(0)-Q(t)$ and equality in (\ref{CS}) for all $i$. If $t>0$, Lemma \ref{2.L} shows that $(a_{11},\dots,a_{nn})$ is a permutation of $(\lambda_1,\dots,\lambda_n)$ giving equality in Hadamard's inequality. Therefore $A$ is diagonal. If $t=0$ then equality in (\ref{CS}) implies that for all $i$ there exist constants $a$ and $b$, not both zero, such that $av_{ji}=bv_{j,\tau(i)}$ for all $j$ such that $\lambda_j\ne 0$. Thus, for all $j$, $\lambda_je_j^*V(ae_i-be_{\tau(i)})=0$ and hence, as above, $e_j^*VA(ae_i-be_{\tau(i)})=0$. This holds for all $j$, and $V$ is invertible, so we conclude that $aAe_i=bAe_{\tau(i)}$, that is, $Ae_i$ and $Ae_{\tau(i)}$ are collinear.
\end{proof}

\begin{rem} The collinearity condition for equality above may be expressed in terms of matrix rank: For $J\subseteq \{1,\dots,n\}$, let $P_J$ be the orthogonal projection onto $\operatorname{span}\{e_j:j\in J\}$. Let $J^\tau_1,\dots,J^\tau_{n_\tau}$ be the orbits of $\tau\in D_n$. Then $I=\sum_{k=1}^{n_\tau}P_{J_k}$ so $A=\sum_{k=1}^{n_\tau}AP_{J_k}$. The condition that the two vectors $Ae_i$ and $Ae_{\tau(i)}$ are collinear for each $i$ is equivalent to saying that the rank of $AP_{J_k}$ is at most 1 for $k=1,\dots,n_\tau$.
\end{rem}

The next result follows from Lemma \ref{2.L} and the Cauchy-Schwarz estimates of Theorem \ref{refined}. We state it without proof.

\begin{thm} Let $\tau\in D_n$. If for all $i$, the $(i,\tau(i))$ entry of $V^*V$ is zero then
$$
\prod_{i=1}^n\lambda_i+\prod_{i=1}^n|b_{i,\tau(i)}|\le\prod_{i=1}^n b_{ii}.
$$
\end{thm}

Next is our proof of the motivating result: Theorem \ref{ZYmain} for positive semi-definite matrices, including conditions for equality.

Observe that if $n=2$, (\ref{maxtau}) reduces to equality for every $A$. 

 \begin{cor}\label{refinedcor} \cite{ZY97} Let  $A=(a_{ij})$ be positive semi-definite matrix, and $\sigma\in S_n$. If $\sigma$ is not the identity permutation, then
 	\begin{equation}\label{maxsigma}
 	\det(A)+\prod_{i=1}^n|a_{i,\sigma(i)}|\le \prod_{i=1}^na_{ii}.
 	\end{equation}
Equality holds if $A$ is diagonal, or if for each $i$ either: $\sigma(i)=i$ and $Ae_i$, $e_i$ are collinear; $\sigma(i)\ne i$ and $\sigma$ is a transposition; or $\sigma(i)\ne i$ and $Ae_i$, $Ae_{\sigma(i)}$ are collinear. If $A$ is positive definite, these conditions are also necessary for equality.
 \end{cor}

\begin{proof} Let $F$ be the set of fixed points of $\sigma$, a proper subset of $\{1,\dots,n\}$, and let $G$ be its complement. If $F$ is empty, the result follows from Theorem \ref{refined}. Note that both $A[F]$ and $A[G]$ are positive semi-definite. Fischer's inequality (see \cite[Theorem 7.8.5]{HJ13}), followed by Hadamard's inequality, gives
\begin{equation}\label{F-H}
\det(A)\le\det(A[G])\det(A[F])
\le\det(A[G])\prod_{i\in F}a_{ii}.
\end{equation}
Note that the restriction of $\sigma$ to $G$  is a permutation of $G$ with no fixed point. By Theorem \ref{refined}, we have
\begin{equation}\label{refG}
\det(A[G])+\prod_{i\in G}|a_{i,\sigma(i)}|\le\prod_{i\in G}a_{ii},
\end{equation}
provided $G$ has at least three elements. We can remove that restriction, however, because (\ref{refG}) becomes equality when $G$ has exactly two elements, it is impossible for $G$ to have exactly one element, and we have excluded the case that $G$ is empty. 
Combining the last two inequalities gives (\ref{maxsigma}).

If $A$ is diagonal then we clearly have equality in (\ref{maxsigma}). Now suppose that for each $i$ either: $\sigma(i)=i$ and $Ae_i$, $e_i$ are collinear; or $\sigma(i)\ne i$ and, if $\sigma$ is not just a transposition, then $Ae_i$, $Ae_{\sigma(i)}$ are collinear. This implies that $A[F]$ is diagonal, and $A$ has (up to reordering of the standard basis) a block diagonal decomposition with blocks $A[F]$ and $A[G]$. Thus we have equality in (\ref{F-H}). The conditions for equality in Theorem \ref{refined} give equality in (\ref{refG}) when $n\ge3$ and equality is trivial when $n=2$ so we have equality in (\ref{maxsigma}).

Now suppose that $A$ is positive definite and
equality holds in (\ref{maxsigma}).
Since $\det(A)>0$, the Fischer inequality shows that $\det(A[F])>0$ and $\det(A[G])>0$. Therefore we have equality in both  (\ref{F-H}) and (\ref{refG}).

Equality in the Fischer inequality from (\ref{F-H})  implies that $A$ has (up to reordering of the standard basis) a block diagonal decomposition with blocks $A[F]$ and $A[G]$ (see \cite[p. 217]{Zha11}).  Equality in the Hadamard inequality from (\ref{F-H}) implies that $A[F]$ is diagonal. Together, these show that if $\sigma(i)=i$, then $Ae_i$ and $e_i$ are collinear. Equality in (\ref{refG}) implies, via Theorem \ref{refined}, that if $G$ has at least three elements and $\sigma(i)\ne i$, then $Ae_i$, $Ae_{\sigma(i)}$ are collinear. If $G$ has fewer than three elements then $\sigma$ can only be a transposition. This completes the proof.
\end{proof}

Recall that the Hadamard product ``$\circ$" is the entrywise product of matrices. So the Hadamard inequality may be written as $\det A\le \det( A \circ I)$ for a positive semi-definite $A$. Theorem \ref{refined} enables us to state the following result.
     \begin{thm} \label{t4} Let $\tau\in D_n$ be a derangement and $P=(p_{ij})$ where $p_{ij}$ is $1$ when $j=\tau(i)$ and zero otherwise. 
For a positive semi-definite matrix $A=(a_{ij})$,	\begin{eqnarray} 
             \det (A\circ I)\ge \det A+|\det (A\circ P)|.
            	\end{eqnarray}   Equality holds if and only if the matrix $A$ is a diagonal matrix or  the two vectors $Ae_i$ and $Ae_{\tau(i)}$ are collinear for each $i$.\end{thm} 
            	                      
          \begin{rem}
         We expect that Theorem  \ref{t4} will stimulate further investigation of Oppenheim-Schur inequalities (see \cite[p. 509]{HJ13}).   \end{rem}
         
  In 1961, Thompson \cite{Tho61} published a remarkable  determinant inequality.

  \begin{thm}\label{Th}If $A=(A_{ij})$ is positive definite with each block $A_{ij}$ square, then 
  	\begin{eqnarray}\label{tho}
   \det A\le \det (\det A_{ij}).
  	\end{eqnarray} Equality holds if and only if $A$ is block diagonal.
  \end{thm}
                 
 We point out an extension of  Theorem \ref{refined} to the block matrix case.  
 
 \begin{thm}If $A=(A_{ij})$ is an $n\times n$ block positive definite matrix with each block $A_{ij}$ square, then for any derangement $\tau\in D_n$,
   	\begin{eqnarray*} 
    \det A+\prod_{i=1}^{n}|\det A_{i,\tau(i)}|\le  \prod_{i=1}^{n}\det A_{ii}.
   	\end{eqnarray*} 
Equality holds if and only if $A$ is block diagonal.
   \end{thm}
    \begin{proof} By (\ref{tho}), it suffices to work with the $n\times n$ positive definite matrix $(\det A_{ij})$. The conclusion then follows by  Theorem \ref{t4}.    
    \end{proof}
  
 \section*{Acknowledgments} Both authors are grateful to Professor Xingzhi Zhan for pointing to the early work of Xiao-Dong Zhang and  Shangjun Yang on this topic. The authors also acknowledge some comments from the referee which  help improve the presentation.   The work of M. Lin is supported by the National Natural
 Science Foundation of China (Grant No. 11601314). The work of G. Sinnamon is supported by the
 Natural Sciences and Engineering Research Council of Canada.

\end{document}